\documentclass[12pt]{amsart}

\usepackage{color}

\usepackage{dsfont}
\usepackage{amsmath}
\usepackage{amssymb}
\usepackage{graphicx}

\usepackage{amssymb}

\usepackage{amsfonts}

\usepackage{soul}

\usepackage{mathpazo}
\usepackage{color}
\usepackage{yfonts}

\usepackage{paralist}

\usepackage{stmaryrd}

\usepackage{amsxtra}

\usepackage{chngcntr}

\def\Id{{\rm Id}}

\def\eps{\varepsilon}

\def \R{{\mathbb R}}

\def \Z{{\mathbb Z}}
\def \N{{\mathbb N}}

\def\de{ \delta  }

\newcommand{\T}{{\mathbb T}}
\newcommand{\prf}{{\begin{proof}}}
\newcommand{\epf}{{\end{proof}}}

\newcommand{\BB}{{\mathbb B}}

\DeclareMathOperator{\diff}{Diff}

\newtheorem{theo}{\sc Theorem}
\newtheorem{prop}{\sc Proposition}%[section]
\newtheorem{cor}{\sc corollary}

\theoremstyle{definition}

\def\bee{\begin{equation}}
\def\eee{\end{equation}}

\theoremstyle{rema}

\newcommand{\tx}{\tilde x}
\newcommand{\ty}{\tilde y}
\newcommand{\tz}{\tilde z}
\newcommand{\tB}{\tilde B}

\newcommand{\tF}{\tilde {\mathbb F}}
\newcommand{\tR}{\tilde R}

\newcommand{\tS}{\tilde S}
\newcommand{\tT}{\tilde T}
\newcommand{\tD}{\tilde \Delta}
\newcommand{\FF}{\mathbb F}

\newcommand{\TT}{{\mathbb T}}

\newcommand{\NN}{{\mathbb N}}

\newtheorem{definition}{Definition}%[section]
\numberwithin{equation}{section}
\renewcommand{\diff}{{\rm Diff}^{d+\eps}_\mu(\BB^d)}

\definecolor{dgreen}{rgb}{0.1,0.6,0.1}

\definecolor{bluegreen}{rgb}{0.1,0.5,0.2}%{0.6,0.3,0.4}
\definecolor{bpurple}{rgb}{0.74,0.2,0.64}%{0.6,0.3,0.4}

\definecolor{orange}{rgb}{0.8, 0.33, 0.0}

\definecolor{orange}{rgb}{1,0.5,0}

\title[Realizing $d$-dim. dynamics by renormalization of $C^d$-perturb. of identity]{Realizing arbitrary $d$-dimensional dynamics by renormalization of $C^d$-perturbations of identity}
\author{Bassam Fayad,  \quad Maria Saprykina}
\begin{document}
\maketitle

\counterwithout{equation}{section}

%\section{Introduction}
\begin{abstract} Any $C^d$ conservative map $f$ of the $d$-dimensional unit ball $\BB^d$ can be realized by renormalized iteration of a $C^d$ perturbation of  identity: there exists a conservative diffeomorphism of $\BB^d$, arbitrarily close to identity in the $C^d$ topology, that has a periodic disc on which the return dynamics after a $C^d$ change of coordinates is exactly $f$. 
\end{abstract} 

\section*{}

What kind of dynamics can be {\it realized by renormalized iteration} of some  diffeomorphism $F$ of the unit ball that is close to identity? 
Given a $d$-dimensional $C^r$-diffeomorphism $F$, its {\it renormalized iteration} is an iteration of $F$, restricted
to a certain $d$-dimensional ball and taken in some $C^r$-coordinates in which the ball acquires radius 1.
%In other words, the above question sounds: what kind of diffeomorphisms are equal to a renormalized iteration of some close to identity diffeomorphism  of the unit ball.

%Being {\it realized by renormalized iteration of }$f$ here means that, after a change of coordinates, the target dynamics can be seen as an iterate of $f$ on some (small) periodic ball. 
%  phenomena on a unit ball $\BB^d$ can be realized by maps close to identity. This question has a physical relevance in the context of the theory of turbulence proposed by Ruelle and Takens. 
This natural question can be traced back to a celebrated paper by Ruelle and Takens
\cite{RT}, where it appeared in connection to the mathematical notion of turbulence. From a subsequent paper by Newhouse, Ruelle and Takens \cite{NRT} it can be seen  that any  dynamics of class $C^d$ on the $d$-dimensional torus $\T^d$ can be realized   by  renormalized iteration  of a 
$C^d$-small perturbation of the identity map
on $\T^d$.\footnote{ In \cite{T} the straightforward strategy of \cite{RT} that leads to the latter result is clearly explained. We will reproduce this sketch below".}
This result is specific to tori.  As D.Turaev points out in \cite{T}, it implies that on an arbitrary manifold $M$ of dimension $d\geq 2$, 
arbitrary $d$-dimensional dynamics can be implemented by iterations of 
$C^{d-1}$-close to identity maps of $\BB^d$, but
the construction gives no clue of whether the same can be said about the $C^{d}$-close to identity maps.

The present note shows that the construction of \cite{NRT} can be enhanced  by an application of a method   
in the spirit of Moser \cite{M} (or Anosov-Katok \cite{AK}),  to get realization by  renormalized iteration of $C^{d+\eps}$-close to identity maps of $\BB^d$.

More precisely,
let  $\BB^d= \{ x\in \R^d : \| x \| \leq 1\}$  denote the unit ball and $\mu$ stand for the Lebesgue measure on it;  
for $r \in \N \cup \{ \infty \}$ we denote by ${\rm
  Diff}^r_\mu(\BB^d)$ the set of diffeomorphisms of class $C^r$ of
$\BB^d$, 
preserving the boundary.

\begin{theo} \label{theo.main} For any natural $d\geq 2$ there exists
$\eps_0>0$ (one can take $\eps_0 = \frac{1}{(d+1)^4}$)
such that the  following holds.
  
 For any $0<\eps <\eps_0$, any $\de >0$, any $f\in {\rm Diff}^{d+\eps}_\mu(\BB^d)$
 there exists $F \in {\rm Diff}^{d+\eps}_\mu(\BB^d)$ and a periodic
 sequence of balls $B,F(B),\ldots,F^M(B)=B\subset \BB^d$ 
such that 
\begin{itemize}
\item $\|F-{\rm Id}\|_{d+\eps}  \leq \de $; 
\item $h \circ F^M \circ h^{-1}=f$, where $h$ is the similarity that sends $B$ to $\BB^d$. 
\end{itemize}
\end{theo}

%%%%%%%%%%%%%%%%%%%%%%%
%%%%%%%%%%%%%%%%%%%%%%%%%
%%%%%%%%%%%%%%%%%%%%%%%%%%%
%{\bgreen Check : We formulated our result in the volume preserving case because it is more stringent to obtain the realizations in this context, however the same statement and the same construction hold for non %conservative dynamics. } \marginpar{Let it be for now? }

It should be mentioned that this theorem does not hold for $d=1$, see e.g., \cite{T}, Sec.1.

We note that in \cite{NRT}, the restriction on smoothness is removed, to the price of realizing $d$-dimensional dynamics by renormalizing $(d+1)$-dimensional maps on some $d$-dimensional  embedded manifold.

Following up on \cite{RT,NRT} and his own work on universality,  Turaev in \cite{T} asks whether an arbitrary $d$-dimensional 
 dynamics can be realized by iterations of a $C^{r}$-close to identity map of $B^d$  and a large $r$? 
  Theorem \ref{theo.main} does not say anything  for $r>d+\eps$. 
  While a more careful application of the same tools of its proof may yield the same statement in $C^{d+1-\eps}$ regularity,  dealing with higher regularities will certainly require new ideas.

\subsection*{Related results}  Note that in this note we are concerned with  {\it exactly realizing }  
and not just {\it approximating} any  given map.  Approximating any given dynamics by renormalization (on almost periodic discs) is a very interesting topic with a vast literature. For instance, D. Turaev in \cite{T}, shows that for any $r \geq 1$ the renormalized iterations of 
$C^r$-close to identity maps of an $n$-dimensional unit ball $\BB^n$ ($n \geq 2$) form a residual set among all orientation-preserving $C^r$-diffeomorphisms $\BB^n \to R^n$. 
As an application he shows that any generic $n$-dimensional dynamical phenomenon can be 
arbitrarily closely (in $C^r$) approximated  by iterations of $C^r$-close to identity maps, with the same dimension of the phase space. 
We refer to Turaev's paper for an account and for references. 
We just mention that recently, a highlight of the universality approach  was the construction by Berger and Turaev in \cite{BT} 
of smooth conservative disc diffeomorphisms that are arbitrarily close to identity in any regularity and that have positive metric entropy, thus solving a conjecture made by Herman in \cite{herman-ICM}.

\subsection*{An application  to universality in the neighborhood of an elliptic fixed point of a area preserving surface diffeomorphisms.}  A. Katok (personal communication) observed, that any area preserving
surface diffeomorphism that has an elliptic periodic point can be
perturbed in $C^r$ topology (arbitrary $r$) so that the perturbed
diffeomorphism is area preserving and has a periodic disc on which the
return dynamic is identity. Hence, we obtain the following consequence
of Theorem \ref{theo.main}.
 \begin{cor}  \label{cor.main} Any $C^{2+\eps}$ area preserving
diffeomorphism $g$ of a surface that has an elliptic periodic point
can be perturbed in the $C^{2+\eps}$ topology into
an area preserving diffeomorphism $\tilde{g}$ that has a periodic disc
on which the renormalized dynamics is equal to any prescribed $F \in
{\rm Diff}^{2+\eps}_\mu(\BB^2)$.  
\end{cor}
In other words, arbitrarily close (in the sense of $C^{2+\eps}$) to any area preserving
diffeomorphism with an elliptic periodic point one can find any prescribed
dynamics.

We do not know whether every area preserving surface diffeomorphism
that is not uniformly hyperbolic can be perturbed in the $C^2$
topology into one that has elliptic periodic points (this is known to
hold in $C^1$ topology). Would this be proved, Corollary
\ref{cor.main} would imply that any area preserving surface
diffeomorphism that is not uniformly hyperbolic can be perturbed in
$C^2$ topology into one that contains any prescribed dynamics.

\bigskip

\section*{Proofs}

Our construction follows in part the straightforward approach of \cite{RT,NRT} of fragmenting the target dynamics into a composition of a large number of close to identity maps that are then reproduced (up to rescaling) on a sequence of pairwise disjoint small balls.  Thus, we start by presenting the constructions from \cite{RT,NRT}, and then introduce and explain the changes we had to make to carry out the construction in slightly higher regularity.

\subsection*{1. Fragmentation}
The first ingredient of the proof is the following proposition of \cite{RT}; see \cite{NRT} or Turaev \cite{T}, page 2, for a comprehensive illustration. 

\begin{prop} [Fragmentation] \label{prop.fragmentation} Given $r\in \N$, $f \in {\rm Diff}^r_\mu(\BB)$,  for any $M \in \N$ there exists $f_0,\ldots,f_{M-1} \in {\rm Diff}^r_\mu(\BB)$ and a constant $C(r,f) >0$ such that 
\begin{itemize}
\item $\|f_i-{\rm Id}\|_r \leq C(r,f) M^{-1}$,
\item $f=f_0 \circ \ldots \circ f_{M-1}$.
\end{itemize}
\end{prop}
\begin{proof}
Consider a Lipschitz isotopy $\psi: [0,1] \to  {\rm Diff}^r_\mu(\BB)$ (Lipschitz in $t \in [0,1]$) such that 
$\psi_0={\rm Id}$ and $\psi_1=f$. Let $f_i=\psi_{i/M}\circ \psi_{(i-1)/M}^{-1}$, for $i=0,\ldots,M-1$. The estimate is straightforward. 
\end{proof}

\subsection*{2. Idea of the proof by \cite{RT,NRT}.}
On a torus $\TT^{d-1}$ consider an $A^{d-2}$-periodic translation 
$$
S^t(X)=X+t(1,\frac{1}{A}, \dots ,\frac{1}{A^{d-2}}).
$$ 
Let $\gamma$ be the  closed invariant curve of $S^{t}$ passing through the origin. 
The turbular neighbourhood of $\gamma$ of radius $\frac{1}{3A}$ does not intersect itself. Let $B$ be the $(d-1)$-dimensional ball of radius 
$\rho := \frac{1}{4A}$ centred at the origin, and let $B_i=S^{\frac{i}{A}}(B)$ for $i=0,\dots, A^{d-1}-1$. 
Then 
$S^t$ is an isometry,
$S^{A^{d-1}}(B)=B$, and
all $B_i$ for $i=0,\dots, A^{d-1}-1$ are disjoint. In other words, $B$ is the base of a periodic tower of discs for the map $S^{\frac{1}{A}}$. The height of the tower is $M=A^{d-1}$.

To prove the theorem of  \cite{NRT}, given $\eps>0$ and  a map $f\in {\rm Diff}^{d-\eps}_\mu(\TT^{d-1})$,  one applies Proposition \ref{prop.fragmentation} with 
$M=A^{d-1}$ to obtain $f_0\dots ,f_{M-1}$ with 
$\|f_i-{\rm Id}\|_{d} \leq C(d,f) M^{-1}$ such that $f=f_0 \circ \ldots \circ f_{M-1}$. Let $h_i$ be a similarity sending 
$B_i$ into the unit ball $\BB^{d-1}$ (i.e., $h_i$ expands linearly by $1/\rho$), and define the desired map by
$F|_{B_i}=S^{\frac{1}{A}} h_i^{-1}\circ f_i\circ h_i$ for $i=0,\dots, M-1$; extend $F$ by identity to the whole $\TT^{d-1}$.  
One easily estimates  
$$
\|F-\Id\|_{d-\eps}\leq \rho^{-(d-1-\eps)} \max_i\|f_i-\Id\|_{d-\eps}\leq (4A)^{(d-1-\eps)} M^{-1}=c_0 A^{-\eps},
$$ 
which is small for large $A$. 

In order to use the same idea for 
$ \BB^d$, one can embed the set $[0,1]\times \T^{d-1}$  into $\BB^d$ and do the same construction (extending  the balls to ${d}$-dimensional ones). Then on $\BB^d$ one gets a realization $\tilde F$ such that
$\|\tilde F-\Id\|_{d-\eps}$ is small, exactly as for the torus 
$\TT^{d-1}$.

\subsection*{3. Permutation map}
To increase the smoothness of the realization for $\BB^d$, we will find a longer sequence of $d$-dimensional balls $B_i$, $i=0,\dots, M-1$, of radius $\rho$,   with 
$$
\rho=\frac1{4A}, \quad M=qA^{d-1},
$$ 
where $q$ is of order $A^{\eps_0}$ for a certain 
$\eps_0=\eps_0(d)$, and a transformation $T$ mapping each $B_i$ into $B_{i+1}$ with $T^M(B_0)=B_0$ (the transformation $T$ plays the role of $S^{\frac{1}{A}}$ in the argument above).  
The increase of $M$ together with a relatively tame estimate for the norm of $T$ permits us to estimate the closeness of approximation in the $C^{d+\eps}$-norm.
We are able to ensure that
$T$ maps $B_i$ into $B_{i+1}$ isometrically not for all, but for {\it a large proportion} of the iterates $i=1,\dots M$, and this is enough for our purposes.

\begin{definition} [Funny periodic tower of balls] For $z \in \BB^d$,
  $\eta>0$ we denote by $B(z,\eta)$ the ball of radius $\eta$ around
  $z$. For $T \in {\rm Diff}^r_\mu(\BB^d)$, $N \in \N$,
  $\eta,\gamma>0$, we say that 
$B=B(z,\eta)$  is a base of an $(\eta,N,\gamma)$-funny periodic tower of discs for $T$ if 
\begin{itemize}
\item All $B_i:= T^i B$, $i=1,\ldots,N-1$, are disjoint;
\item $T^N B=B$, and  $T^N : B \to B$ is the Identity map,
\item at least $[\gamma N]$ integers $n_i \in [1,N]$ are isometry times for $T$ in the sense that $T^{n_i}$ is an isometry from $B$ to the disc $B_i:=T^{n_i}B$. 
\end{itemize}
\end{definition} 

The terminology {\it funny} periodic towers is borrowed from the notion of funny rank one introduced by J.-P. Thouvenot to weaken the notion of rank one systems (see \cite{Ferenczi}).

\begin{prop} [Permutation map]  \label{prop.tower} 
For any $d\geq 2$, $r\geq 1$, $A ,q\in \N$ such that $q^{r^4}  \ll A$,  
there exists $T \in
{\rm Diff}^\infty_\mu(\BB^d)$ such that 
\begin{itemize}
\item $\|T - {\rm Id}\|_r \leq \frac{q^{r^4} }{A}$,
\item $T$ has a $(\frac1{4A}, \frac12 q A^{d-1}, 1/2)$-funny periodic tower of discs.
\end{itemize}
\end{prop}

\begin{proof}

\noindent {\sc  Embedding  a ``fat torus" $\FF$ into $\BB^d$.}
For a fixed $d\geq 2$, 
denote by $\FF$  a ``fat torus": 
$$
\FF= [0,1]\times \T^{d-1},
$$ 
where $\T^{d-1}=\R^{d-1}/\Z^{d-1}$. We will denote the natural Haar-Lebesgue measure on $\FF$  by 
$\mu$, and the coordinates in
$\FF$ by $(x,y,z)\in [0,1]\times \T \times \T^{d-2}$. 
(In two dimensions
a ``fat torus'' $\FF^2$ is an annulus with coordinates $(x,y)$.)

Given $q>1$, consider an open set $P \in \BB^d$ such that 
\begin{itemize}
\item $P$ contains a cube $Q_q:=(0,1-\frac1{q})^{d}$;
\item  $P$ is $C^{\infty}$-diffeomorphic to $\FF$;
\item $\mu(P) =1$.
\end{itemize}
Let   $h: P \to \FF$ be a  $C^\infty$ diffeomorphism, transforming the Lebesgue measure on 
$\BB^{d}$ to that on $\FF$,
and such that $h$ is an isometry on the cube $Q_q$.
Such a volume-preserving map exists by \cite{M} or \cite{AK}.

\medskip

\noindent {\sc Partition of $\FF$.}
Let 
$$
\begin{aligned}
R_q&=[0,1] \times \left [0,\frac{1}{q} \right] \times \T^{d-2} \subset \FF,\\
\Delta_q &= \left[ 0.1\, ,0.9 \right] \times 
\left[\frac{1}{10q}\, , \frac{9}{10q} \right] \times\T^{d-2} \subset R_q.
\end{aligned}
$$ 
Extend this
construction $\frac{1}{q}$-periodically in $y$ to the whole $\FF$ to obtain a partition of $\FF$ into thin rectangular blocks.

\medskip
 
\noindent {\sc Switching to a "longer" fat torus $ \tF_q$.}
Consider another  fat torus 
$$
\tF_q = \left[0,\frac{1}{q} \right] \times  (\R/(q\Z))  \times  \T^{d-2}
$$ 
with coordinates $(\tx,\ty ,\tz)$ corresponding to the above splitting. Let 
$$
\begin{aligned}
\tR_q&= \left[0,\frac{1}{q} \right] \times [0,1] \times \T^{d-2}
\subset \tF_q,  \\
\tD_q&= \left[\frac{1}{10q}\, , \frac{9}{10q} \right] \times
\left[ 0.1\, ,0.9 \right] \times \T^{d-2} \subset \tR_q .
\end{aligned}
$$ 
Notice that $\tD_q$  can be mapped into  $\Delta_q$ by an isometry.
Extend this construction 1-periodically in $\ty$ to the whole $\tF_q$.

We define the circle actions on $\FF$ and $\tF_q$, respectively:
$$
\begin{aligned}
S_t (x ,y, z_1\dots z_{d-2}) &= (x ,y, z_1\dots z_{d-2})+
t\left(0,1, \frac{1}{A}, \dots , \frac{1}{A^{d-2}}\right),\\
\tS_t (\tx,\ty ,\tz_1\dots \tz_{d-2})&=
(\tx,\ty ,\tz_1\dots \tz_{d-2}) +t \left(0,1, \frac{1}{qA}, \dots , 
\frac{1}{qA^{d-2}}\right).
\end{aligned}
$$ 
We now define $h_q$ to be a volume-preserving map from $\FF$ to $\tF_q$ such that 
\begin{itemize}
\item $h_q$ maps the set $\FF|_{y=0}$ to $\tF_q|_{\ty=0}$;
\item $h_q$ acts as identity on the last $d-2$ components; 
\item $h_q(R_q)=\tR_q$, \ $h_q \circ S_{\frac{1}{q}} = \tS_1 \circ h_q$;
\item $h_q(\Delta_q)=\tD_q$, and $h_q$ acts as an isometry in restriction to  $\Delta_q$;
\item $\|h_q\|_r \leq q^{2r}$.
\end{itemize}

\medskip
 
\noindent {\sc Definition of the desired map $T$.}

Let $\varphi : (0,q^{-1})\to \R$ of class $C^r$ be such that
$$
\varphi(\tx)=
\begin{cases}
1/A, & \text{ for } \tx\in [0.3q^{-1},0.7q^{-1}] \\ 
0 & \text{ for } \tx \in (0,0.2q^{-1}] \cup [0.8q^{-1}, q^{-1}).
\end{cases}
$$ 
It is easy to construct such a function satisfying $\| \varphi \|_r\leq q^{2r}/A$. 

Define on $\tF_q$ the shear map
$$\tT_\varphi (\tx,\ty ,\tz_1\dots \tz_{d-2})=
(\tx,\ty ,\tz_1\dots \tz_{d-2}) +\varphi(\tx)\cdot \left(0,1, \frac{1}{qA}, \dots , \frac{1}{qA^{d-2}}\right).
$$
Finally, let $T:P \mapsto P$ be defined by
$$
T=h^{-1}\circ h_q^{-1} \circ \tT_{\varphi} \circ h_q \circ h.
$$
The fact that $\tT_{\varphi}$ equals identity close to the boundary of $\tF_q$ implies that
$T$ equals identity close to the boundary of $P$, and can thus be
extended by identity to the whole $\BB^{d}$.

The obtained transformation  $T:\BB^{d} \mapsto \BB^{d}$ satisfies the conclusion of Proposition \ref{prop.tower}. 
Indeed, we have that 
$$
\|T-{\rm Id}\|_r\leq \left( \|h_q\|_{r+1} \|h_q^{-1}\|_{r+1} \|h\|_{r+1}\|h^{-1}\|_{r+1}\right)^{r+1} \|\tT_\varphi-{\rm
  Id}\|_{r+1} \leq \frac{q^{r^4}}{A}.
$$

As for the funny periodic tower, let $\tB$ be the disc of radius $\rho=\frac1{4A}$ centred at $(\tx,\ty,\tz)=(\frac1{2q},0,0)$. Take for the base of the tower of $T$ the ball
$B:=h^{-1}\circ h_q^{-1}(\tB)$. Under $\tT_\varphi$, the center of $\tB$ is translated
by $1/A$ along the trajectory of the shift $\tS$ passing through
this point. It is easy to see that the tubular neighborhood of radius $1/(3A)$ of this trajectory
does not intersect itself, and $\tS^{qA^{d-1}}=\text{\,
  Id}$. Moreover, $\tT_\varphi$ is an isometry on this neighborhood. 
  Therefore,  $\tB$ is the base of a tower by
discs of height $qA^{d-1}$ for $\tT_\varphi$ on $\tF$. Moreover the times $n_i
\in [0, A^{d-1}q-1]$ such that $\tT_\varphi^{n_i}(\tF ) \in \cup_{\ell=0}^{q-1}
\tS_\ell \tD_q$ represent a proportion of around $8/10$ of $A^{d-1}q$,
hence clearly more than $1/2$ of $n \in [0,A^{d-1}q]$ are isometric times
for $T$. For such times $T^{n_i}|_{B}=h^{-1}\circ h_q^{-1} \circ \tT_\varphi^{n_i}
\circ h_q\circ h$ is a composition of several isometries and is thus an
isometry.

\end{proof}

\subsection*{4. Proof of Theorem \ref{theo.main}}

\begin{proof}[Proof of Theorem \ref{theo.main}]
Given $d\geq 2$, let $\eps_0=\frac{1}{(d+1)^{4}}$. Fix any $0<\eps<\eps_0$, let $r:=d+\eps$, and
assume that $f \in {\rm Diff}^r_\mu(\BB)$. 

Choose a large $q\in \NN$ and $A>q^{r^4} $, let $M=\frac12qA^{d-1}$ and 
apply Proposition \ref{prop.fragmentation} to get
$f_0,\ldots, f_{M-1} \in \diff$ such that 
$\|f_i-{\rm Id}\|_r\leq C(d, r,f) M^{-1}$ and  $f=f_0 \circ \ldots \circ f_{M-1}$.

Let $T$ be as in Proposition \ref{prop.tower} with $d$, $r$, $A$ and $q$ as above. 
We let $B_i=T^{n_i} B$, where $B$ is the base of the $(1/(4A), 2M , 1/2)$-funny periodic tower of discs, 
and  $0=n_0 ,\ldots,n_{M}=2M$  are isometric times for $T$. 
By Proposition \ref{prop.tower}, $T^{n_{M}}=T^{2M}$ is the Identity from
$B$ to $B$, and 
\begin{equation} \label{eq.T} 
\|T-{\rm Id}\|_r \leq  \frac{q^{r^{4}}}{A} .
\end{equation}  
Let $h_{i} :
B_i \to \BB^{d}$ 
be similarities that send $B_i$ onto $\BB^{d}$ such that for $i=0,\ldots,M-1$ we have
\begin{equation} \label{eq.hii}
h_{i} = h_0 \circ T^{-n_{i}}, 
\end{equation}
or equivalently 
\begin{equation} \label{eq.hi}
h_{i+1} = h_i \circ T^{-(n_{i+1}-n_i)}. 
\end{equation}
Observe that since $T^{n_{M}} : B_0 \to B_0$ is the Identity map, then it follows that 
\begin{equation} \label{eq.hN} 
h_{n_M} =h_{0} \circ T^{-n_{M}}=h_0. 
\end{equation}

We define $\bar{F} \in \diff$ such that  
\begin{itemize} 
\item $\bar{F}|_{B_i}= h_i^{-1} \circ f_i \circ h_i$, $\forall
  i=0,\ldots,M-1$;

\item $\bar{F}$ equals Identity on the complementary of $B_0 \sqcup \ldots \sqcup B_{M-1}$. 

\end{itemize} 

Observe that since $h_i$ are similarities that expand by $1/\rho=4A$, we have that 
\begin{equation} \label{eq.norm} 
\begin{aligned}
\|\bar{F} - {\rm Id}\|_{r} \leq & c_0 A^{r-1} \max_i \|f_i - {\rm
  Id}\|_{r} \leq  A^{r-1}C(r,d,f) M^{-1} = \\
  &C(r,d,f) q^{-1}A^{d-r} 
\end{aligned}
\end{equation}

Define finally ${F} \in \diff$ by
$$
F=T \circ \bar{F}.
$$ 
Since 
$$
F-{\rm Id}=(\bar{F} - {\rm Id} ) + (T - {\rm Id} )\circ \bar{F} ,
$$
estimates  \eqref{eq.T} and \eqref{eq.norm} 
imply that for a constant $C _1(d,r) $ we have
$$
\|F-{\rm Id}\|_r  \leq 
\|\bar{F} - {\rm Id}\|_{r}  +  \| T - {\rm Id}\|_{r} \leq 
C(r,d,f) q^{-1} A^{r-d}   + C _1(d,r) \frac{q^{r^{4}}}{A}  .
$$
Given $\de>0$, we need to identify those $\eps$ for which the 
latter sum can be made smaller than $\de$ by a choice of $A$ and $q$.
Notice that $q$ enters in the latter two terms both in the numerator
and in the denominator.
Let $c_1$, $c_2$ be positive constants such that $c_1+c_2=1$. 
We need 
$$
C(r,d,f) q^{-1} A^{\eps }\leq c_1\de,\quad  C _1(d,r) \frac{q^{r^{4}}}{A} \leq c_2\de,
$$
which gives 
\begin{equation}\label{est_q}
C_2(r,d,f) A^{\eps } (c_1\de)^{-1}  < q < (c_2 \de A )^{1/r^{4}}.
\end{equation}
Such a choice of $q$ is possible if 
$$
\eps < \frac{1}{r^{4}} =\frac{1}{(d+\eps)^{4}}:=\eps_0
$$
and $A$ is sufficiently large. Let us summarise the above argument.  Assume that $\eps $ satisfies the above inequality, let 
$r=d+\eps$. Given $\delta$, choose a sufficiently large $A$ and let $q$ satisfy \eqref{est_q}. Then we have $\| F - {\rm Id}\|_{d+\eps} <\delta$, as desired.
 
At the same time, $F^{n_{M}}(B_0)=B_0=B$ and by \eqref{eq.hi} and \eqref{eq.hN} we have 
$$
F^{n_{M}}|_{B}= h_0^{-1} \circ f \circ h_0,
$$  
which completes the proof of Theorem \ref{theo.main}. 
\end{proof}

\end{document}